\documentclass[reqno,12pt] {amsart}
\usepackage{amsmath, mathrsfs}
\usepackage[usenames, dvipsnames]{color}

\newtheorem{theorem}{Theorem}[section]
\newtheorem{lemma}[theorem]{Lemma}

\theoremstyle{definition}
\newtheorem{definition}[theorem]{Definition}

\newtheorem{example}[theorem]{Example}

\theoremstyle{remark}
\newtheorem{remark}[theorem]{Remark}

\newcommand{\mysection}[1]{\section{#1}
\setcounter{equation}{0}}

\newcommand{\bR}{\mathbb R}

\newcommand\md{\mathrm{d}}

\newcommand\cP{\mathscr{P}}

\newcommand{\Div}{\operatorname{div}}

\makeatletter
\def\dashint{\operatorname%
{\,\,\text{\bf--}\kern-.98em\DOTSI\intop\ilimits@\!\!}}
\makeatother

\renewcommand{\epsilon}{\varepsilon}

\begin{document}
\title [The aggregation equation]{On similarity solutions to the multidimensional aggregation equation}

\author[H. Dong]{Hongjie Dong}
\address[H. Dong]{Division of Applied Mathematics, Brown University, 182 George Street, Box F, Providence, RI 02912, USA}
\email{Hongjie\_Dong@brown.edu}
\thanks{Hongjie Dong was partially supported by the National Science Foundation under agreement No. DMS-0800129.}




\begin{abstract}
We study similarity solutions to the multidimensional aggregation
equation $u_t+\Div(uv)=0$, $v=-\nabla K*u$ with general power-law
kernels $K$ such that $\nabla K(x)=x|x|^{\alpha-2},\alpha\in
(2-d,2)$. We analyze the equation in different regimes of the
parameter $\alpha$. In the case when $\alpha\in [4-d,2)$, we give
a characterization of all the ``first kind'' radially symmetric
similarity solutions. We prove that any such solution is a linear
combination of a delta ring and a delta mass at the origin. On the
other hand, when $\alpha\in (2-d,4-d)$, we show that there exist
multi delta-ring similarity solutions in $\bR^d$. In particular,
our results imply that multi delta-ring similarity solutions exist
in 3D if $\alpha$ is just a little bit below $1$.
\end{abstract}

\maketitle

\mysection{Introduction}               \label{sec1}

In this paper, we study the multidimensional aggregation equation
\begin{equation}
                            \label{Agg}
u_t+\Div(uv)=0,\quad v=-\nabla K*u
\end{equation}
for $x\in\bR^d$ and $t> 0$ with the initial data
\begin{equation*}
u(0,x)=u_0(x),\quad x\in\bR^d.
\end{equation*}
Here $d$ is the space dimension, $u\ge 0$ is the density function, $v$ is the velocity field, $K$ is the interaction potential kernel which typically is taken to be radially symmetric, and $*$ denotes the spatial convolution. The aggregation equation arises in various models for biological aggregation and problems in granular media; see, for instance, Mogilner and Edelstein-Keshet \cite{ME} and Carrillo, McCann, and Villani \cite{CMV}. The equation appears as an alternative
to the Debye--Huckel or Keller--Segel model.
It can be viewed as a continuum model for $N$ particles $X_1,\ldots,X_N$, which interact via a pairwise interaction potential:
$$
\frac{dX_i}{dt}=-\sum_{k=1}^N m_k \nabla K(X_i-X_K).
$$
Examples of such kernels which are commonly used are $K(x)=1-e^{-|x|}$ and $K(x)=|x|^\alpha$.
For nonsmooth kernels $K$, a typical property of the solutions is that they concentrate mass in finite time. On the other hand, for smooth kernels, solutions blow up at the infinity (cf. \cite{BCL}).

In the past few years, the aggregation equation has attracted many attentions due to its diverse applications in biology and physics. The problems of the well-posedness in various spaces, the existence of finite-time blowups, asymptotic behaviors of solutions of this equation have been studied extensively in a series of papers; see, for example, Li and Toscani \cite{LT}, Bodnar and Vel\'azquez \cite{BV}, Bertozzi and Laurent \cite{BL}, Laurent \cite{La}, Bertozzi, Laurent, and Rosado \cite{BLR}, Carrillo et al. \cite{CFFLS}, the author \cite{Don10}, Bertozzi, Garnett, and Laurent \cite{BGL} and references therein. The equation with an additional dissipation term is closely relevant to the Keller--Segel model and was investigated by Li and Rodrigo in \cite{LR1,LR2}. See also \cite{BL2} for a nice review about this equation.

We are interested in profiles of first-kind self-similar solutions to the equation \eqref{Agg} with general power-law interaction kernel $K(x)=|x|^\alpha/\alpha, \alpha\in (2-d,2)$ (when $\alpha=0$, naturally we take $K(x)=\log|x|$). These solutions are mass conserving and of the form
\begin{equation}
                                    \label{eq00.25}
u(t,x)=\frac 1 {R(t)^d} u_0\left(\frac x {R(t)}\right).
\end{equation}
For many evolutionary equations, this type of solutions are of special interest as they usually give large time behaviors of global solutions or asymptotic behaviors of finite-time blowups evolving from general initial data.
One class of such solutions is the trivial single delta-ring solutions, where $u_0=m \delta_{|x|=\rho}$ for $m,\rho>0$, i.e., the support of the solution concentrated on a sphere with shrinking radius. Recently in \cite{BCL}, Bertozzi, Carrillo and Laurent studied nontrivial similarity solutions to \eqref{Agg} with $\alpha=1$. In 1D, they constructed a weak solution with support on an open interval:
$$
u(t,x)=\frac 1 {T^*-t} U\left(\frac x {T^*-t}\right),
$$
where $U$ is the uniform distribution on $(-1,1)$. While in 2D, they found a double delta-ring solution given by
$$
u_0=m_1 \delta_{|x|=\rho_1}+m_2\delta_{|x|=\rho_2}
$$
for appropriate $m_i,\rho_i>0$. In the same paper, using a relation between the kernel $K$ and the Laplace operator, they also proved that in any {\em odd} dimension $d\ge 3$ such solutions cannot exist with support on open sets. Later in \cite{Don10}, by observing certain concavity property of the kernel in the polar coordinates, we completely characterized all the radially symmetric first-kind self-similar solutions in any dimension $d\ge 3$ when $\alpha=1$: all such solutions must have the form \eqref{eq00.25}, where $R(t)$ is a linear positive decreasing function and
\begin{equation*}
u_0=m_0\delta_{|x|=0}+m_1\delta_{|x|=\rho_1}
\end{equation*}
for any $m_0,m_1\ge 0,\rho_1>0$. In particular, this result
implies that with $\alpha=1$ the equation \eqref{Agg} does not
have any double or multi delta-ring solutions when $d\ge 3$. We
also mention that in a recent interesting paper \cite{HB}, Huang
and Bertozzi provided very precise numerical evidence of the local
existence of smooth radially symmetric similarity solutions to
\eqref{Agg} when $\alpha=1$ in high space dimensions, which are of
the second kind and do not conserve mass. Similar numerical
simulations for $\alpha\in (0,2)$ have been carried out by the
same authors in a recent preprint \cite{Hua10}, in which they also
studied the long time behavior of global smooth similarity
solutions when $\alpha>2$.

The objective of this paper is to investigate similarity solutions
to \eqref{Agg} with general power-law kernel by using analytic
methods. To the best of our knowledge, except for some numerical
results mentioned above, this problem has not been studied in
detail. We analyze the equation in different regimes of the
parameter $\alpha$. Since $\nabla K=x|x|^{\alpha-2}$, the velocity
field is always contracting. In the case when $\alpha\in [4-d,2)$,
we give a characterization of all the first-kind radially
symmetric similarity solutions. We prove that any such solution is
a linear combination of a delta ring and a delta mass at the
origin; see Theorems \ref{thm2.1} and \ref{thm4.1}. This result
generalizes the aforementioned results in \cite{BCL} and
\cite{Don10} proved for the case $\alpha=1$. On the other hand,
when $\alpha\in (2-d,4-d)$, we show that there exist multi
delta-ring similarity solutions in $\bR^d$; see Theorem
\ref{thm3.1}. This, in particular, implies the existence of multi
delta-ring similarity solutions in 3D if $\alpha$ is just a little
bit below $1$. Finally, we present an example of first-kind
similarity solutions with support on open sets in the limiting
case $\alpha=2-d$, which generalizes the 1D example in \cite{BCL}
with $\alpha=1$; see Example \ref{ex5.1}.

All the proofs in this paper are based on delicate analysis of the
kernel $\phi$ in the polar coordinates; see \eqref{eq16.38} for
the definition of $\phi$. For the proofs of Theorems \ref{thm2.1}
and \ref{thm4.1}, first we note that the arguments in \cite{BCL}
and \cite{Don10} break down in several places for general
$\alpha$. In particular, there is no clear relation between $K$
and the Laplace operator, and we do not have the concavity of the
kernel $\phi$ in the polar coordinates; cf. Remark \ref{rem5}.
Here we refine an idea in \cite{Don10} by observing that instead
of the concavity of $\phi$, the proof of the corresponding result
in \cite{Don10} only uses the monotonicity property of $\phi/r$.
This enables us to extend the argument to the general $\alpha\in
[4-d,2)$. However, we are not able to find a unified proof of this
monotonicity property. We treat the case $\alpha=4-d$ or
$\alpha>5-d$ in Section \ref{sec2} and the case $\alpha\in
(4-d,5-d)$ in Section \ref{sec4} separately by using completely
different methods. The proof of the existence of multi delta-ring
solutions (Theorem \ref{thm3.1}) relies on the fact that
$\phi'(0)<\phi(1)$ when $\alpha\in (2-d,4-d)$. Once this is
verified, we reduce the problem to the existence of positive
solutions to a system of linear algebraic equations.


Regarding the double and multi delta-ring similarity solutions constructed in this paper and \cite{BCL}, an interesting question would be whether these solutions describe the asymptotic behavior of certain finite-time blowup solutions. Another relevant problem is the stability of these  solutions. We shall study these problems in future work.

\mysection{Characterization of similarity solutions}        \label{sec2}

Let $u$ be a radially symmetric first-kind similarity solution
given by \eqref{eq00.25}. Recall that we take $\nabla
K=x|x|^{\alpha-2}$. By the homogeneity, it is easily seen that
$$
v(t,x)=R(t)^{\alpha-1}v_0\left(\frac x {R(t)}\right),\quad v_0=-x|x|^{\alpha-2}*u_0.
$$
Upon using a similar argument as in  \cite{BCL} by looking at the flow map driven by $v$, we know that on the support of $u_0$,
\begin{equation}
                                    \label{eq16.08}
v_0=-\lambda x
\end{equation}
with some constant $\lambda>0$.
Indeed, the function $u$ given by \eqref{eq00.25} is a first-kind similarity solution of \eqref{Agg} {\em if and only if} for any $x$ on the support of $u_0$, the flow map starting from $x$ is equal to  $R(t)x$ at time $t$, which clearly implies \eqref{eq16.08}.
Moreover, since $R(t)$ satisfies the ordinary differential equation
$$
R'(t)=v(t,R(t))=-\lambda R(t)^{\alpha-1}
$$
with the initial condition $R(0)=1$, we have
$$
R(t)=(1-t/T_0)^{1/(2-\alpha)},
$$
where $T_0=(2-\alpha)^{-1}\lambda^{-1}$ is the first time when all the mass concentrates at the origin.

We denote $\cP(\bR^d)$ to be the set of all probability measures on $\bR^d$.
\begin{definition}
Let $\mu\in \cP(\bR^d)$ be a radially symmetric probability measure. We define $\hat \mu\in \cP([0,+\infty))$ by
$$
\hat \mu(I)=\mu(\{x\in \bR^d: |x|\in I\})
$$
for all Borel sets $I$ in $[0,+\infty)$.
\end{definition}

In this section, we give a characterization of all the radially symmetric first-kind similarity measure solutions to \eqref{Agg} in the case that $d+\alpha\ge 5$ or $d+\alpha=4$. We state the result as the following theorem, which generalizes Theorem 3.1 of \cite{Don10} proved for the case $\alpha=1$.

\begin{theorem}[Characterization of similarity solutions]
                                    \label{thm2.1}
Let $d\ge 3$ and $\nabla K(x)=x|x|^{\alpha-2}$. Assume that either $\alpha\in [5-d,2)$ (and thus $d\ge 4$) or $\alpha=4-d$. Then any radially symmetric first-kind similarity measure solution is of the form
$$
\mu_t(x)=\frac 1 {R(t)^d} \mu_0\left(\frac x {R(t)}\right),
$$
where
\begin{equation*}
\hat \mu_0=m_0\delta_0+m_1\delta_{\rho_1}
\end{equation*}
for some constants $m_0,m_1\ge 0$ and $\rho_1>0$.
\end{theorem}

Before we give a proof of the theorem, first we recall that for any radially symmetric measure solution $\mu_t$ and $t>0$,
\begin{equation}
                                            \label{eq16.39}
v(t,x)=-\int_0^\infty \phi(|x|/\rho)\rho^{\alpha-1}\,\md \hat\mu_t(\rho)\frac {x} {|x|},                             \end{equation}
where
$\phi:[0,\infty)\to \bR$ is a function  defined by
\begin{equation}
                            \label{eq16.38}
\phi(r)=\dashint_{S_1}\frac {r-y_1} {|re_1-y|^{2-\alpha}}\,\md \sigma(y);
\end{equation}
See, for example, \cite{BCL} or \cite{Don10}.
From the dominated convergence theorem, it is easily seen that, when $\alpha\in (2-d,2)$, the function $\phi$ is in fact continuous on $[0,\infty)$ and satisfies
\begin{equation}
                                    \label{eq11.30}
\lim_{r\to \infty}\phi(r)/r=0.
\end{equation}

Our proof of Theorem \ref{thm2.1} is founded on the following key lemma.

\begin{lemma}
                                    \label{lem2.2}
i) Let  $d\ge 4$, $\alpha\in [5-d,2)$, and $\nabla K(x)=x|x|^{\alpha-2}$. Then the function $\phi(r)/r$ is a strictly decreasing function on $(0,\infty)$.

ii) Let $d\ge 3$, $\alpha=4-d$, and $\nabla K(x)=x|x|^{\alpha-2}$. Then the function $\phi(r)/r$ is a decreasing function on $(0,\infty)$. Moreover, it is a constant on $(0,1]$ and is strictly decreasing on $[1,\infty)$.
\end{lemma}
\begin{remark}
                            \label{rem5}
The proof of Theorem \ref{thm3.1} of \cite{Don10} uses the concavity of $\phi$ when $\alpha=1$; see Lemma 3.2 \cite{Don10}. This property, however, is unavailable for general power $\alpha\in (2-d,2)$ even under the condition $d+\alpha>5$ or $d+\alpha=4$. Indeed, when $\alpha<1$ we have from \eqref{eq16.38} that $\phi\to 0$ as $r\to \infty$. Since $\phi$ is strictly positive on $(0,\infty)$, it cannot be concave on $(0,\infty)$ in this case.
\end{remark}

\begin{proof}[Proof of Lemma \ref{lem2.2}]
We shall prove that under the conditions of the lemma, $\phi/r$ is a $C^1$ function on $(0,\infty)$ and its derivative $\phi'/r-\phi/r^2$ is negative on $(0,\infty)$.
To this end, we rewrite \eqref{eq16.38} as
\begin{equation}
                                    \label{eq16.59}
\phi(r)=\frac {\omega_{d-1}}{\omega_d}\int_0^\pi \frac {(r-\cos\theta)(\sin\theta)^{d-2}}{A^{2-\alpha}(r,\theta)}\,\md \theta,
\end{equation}
where
$$
A(r,\theta)=(1+r^2-2r\cos\theta)^{1/2}
$$
and $\omega_d$ is the surface area of the unit sphere $S_1$ in $\bR^d$.
For $r\in [0,1)\cup (1,\infty)$, a direct computation gives
\begin{equation}
                                \label{eq17.02}
\phi'(r)=\frac {\omega_{d-1}}{\omega_d}\int_0^\pi (\sin\theta)^{d-2}\left(\frac 1 {A^{2-\alpha}}-(2-\alpha)\frac {(r-\cos\theta)^2}{A^{4-\alpha}}\right)\,\md \theta,
\end{equation}
and
\begin{align}
                                \label{eq17.06}
r&\phi'(r)-\phi(r)\nonumber\\
&=\frac {\omega_{d-1}}{\omega_d}\int_0^\pi (\sin\theta)^{d-2}\left(\frac {\cos\theta} {A^{2-\alpha}}-(2-\alpha)r\frac {(r-\cos\theta)^2}{A^{4-\alpha}}\right)\,\md \theta\nonumber\\
&:=\frac {\omega_{d-1}}{\omega_d}(I_1+I_2).
\end{align}
Integration by parts yields
\begin{align}
                                    \label{eq13.09}
I_1&=\int_0^\pi (\sin\theta)^{d-2}\frac {\cos\theta} {A^{2-\alpha}}\,\md \theta=\int_0^\pi \frac {\big((\sin\theta)^{d-1}\big)'} {(d-1)A^{2-\alpha}}\,\md \theta\nonumber\\
&=r\int_0^\pi \frac{2-\alpha}{d-1}\cdot \frac {(\sin\theta)^d} {A^{4-\alpha}}\,\md \theta.
\end{align}
Clearly,
\begin{align*}
I_2&=r\int_0^\pi -(2-\alpha)(\sin\theta)^{d-2}\frac {(r-\cos\theta)^2}{A^{4-\alpha}}\,\md \theta\\
&=r\int_0^\pi -(2-\alpha)(\sin\theta)^{d-2}\frac {A^2-(\sin\theta)^2}{A^{4-\alpha}}\,\md \theta.
\end{align*}
Thus, by \eqref{eq17.06},
\begin{align}
                                \label{eq11.46}
r&\phi'(r)-\phi(r)\nonumber\\
&=\frac {r(2-\alpha)\omega_{d-1}}{(d-1)\omega_d}\int_0^\pi \left(d\frac {(\sin\theta)^d} {A^{4-\alpha}}-(d-1)\frac {(\sin\theta)^{d-2}} {A^{2-\alpha}}\right)\,\md \theta.
\end{align}
We shall first consider Assertion ii), then Assertion i).

{\em Assertion ii): $d+\alpha=4$.} In this case, from \eqref{eq11.46} we have
$$
r\phi'(r)-\phi(r)=\frac {r(2-\alpha)\omega_{d-1}}{(d-1)\omega_d}(d\Psi_d-(d-1)\Psi_{d-2}),
$$
where
\begin{equation}
                                \label{eq17.27}
\Psi_k=\Psi_k(r):=\int_0^\pi \frac {(\sin\theta)^k} {A^{k}}\,\md \theta,\quad k=0,1,2,...
\end{equation}
It is known that $\Psi_k$ can be computed explicitly:
$$
\Psi_k(r)= \left\{\begin{aligned}
&\frac {\omega_{k+2}}{\omega_{k+1}} \quad &\text{on}\quad [0,1],\\
&\frac {\omega_{k+2}}{\omega_{k+1}}r^{-k} \quad &\text{on} \quad (1,\infty).
\end{aligned}\right.
$$
Indeed, this is obviously true when $k=0$. For $k\ge 1$, $\omega_{k+1}\Psi_k(r)$ is the integral of $|x-re_1|^{-k}$ on the unit sphere in $\bR^{k+2}$. Notice that for fixed $r>1$, $|x-re_1|^{-k}$ is a harmonic function on the unit ball in $\bR^{k+2}$. Therefore, by the mean value theorem, we get $\omega_{k+1}\Psi_k(r)=\omega_{k+2}r^{-k}$. This also holds true when $r=1$ by the dominated convergence theorem. The remaining case $r\in (0,1)$ follows from the simple identity $A(r,\theta)=rA(1/r,\theta)$.
Since $\omega_{k+2}=2\pi \omega_k/k$, we get $r\phi'-\phi=0$ on $[0,1)$ and $r\phi'-\phi<0$ on  $(1,\infty)$.
Note that because
$|r-\cos\theta|\le A$ and  $|\sin\theta|\le A$,
the integral on the right-hand side of \eqref{eq17.06} is absolutely convergent and continuous at $r=1$ by the dominated convergence theorem. Thus \eqref{eq17.06} holds on the whole region $[0,\infty)$. So we conclude $r\phi'-\phi=0$ on $[0,1]$ and $r\phi'-\phi<0$ on  $(1,\infty)$.

{\em Assertion i): $d+\alpha\ge 5$.} In this case, we estimate $I_2$ in a different way. For $r\in [0,1)\cup (1,\infty)$, integration by parts gives
\begin{align*}
I_2&=-r(2-\alpha)\int_0^\pi (\sin\theta)^{d-2}\frac {r^2-2r \cos\theta+(\cos\theta)^2}{A^{4-\alpha}}\,\md \theta\\
&=-r(2-\alpha)\int_0^\pi (\sin\theta)^{d-2}\frac {r^2}{A^{4-\alpha}}+\frac{\big((\sin\theta)^{d-1}\big)'}{(d-1)A^{4-\alpha}}(-2r+\cos\theta)\,\md \theta\\
&=-r(2-\alpha)\int_0^\pi (\sin\theta)^{d-2}\frac {r^2}{A^{4-\alpha}}+\frac{(\sin\theta)^{d}}{(d-1)A^{4-\alpha}}\\
&\qquad +\frac{(4-\alpha)r(\sin\theta)^{d}}{(d-1)A^{6-\alpha}}(-2r+\cos\theta)\,\md \theta.
\end{align*}
Integrating by parts again in the last term above, we get
\begin{align*}
I_2&=-r(2-\alpha)\int_0^\pi (\sin\theta)^{d-2}\frac {r^2}{A^{4-\alpha}}+\frac{(\sin\theta)^{d}}{(d-1)A^{4-\alpha}}\\
&\qquad -\frac{2(4-\alpha)r^2(\sin\theta)^{d}}{(d-1)A^{6-\alpha}}+
\frac{(4-\alpha)r\big((\sin\theta)^{d+1}\big)'}{(d-1)(d+1)A^{6-\alpha}}\,\md \theta\\
&=-r(2-\alpha)\int_0^\pi (\sin\theta)^{d-2}\frac {r^2}{A^{4-\alpha}}+\frac{(\sin\theta)^{d}}{(d-1)A^{4-\alpha}}\\
&\qquad -\frac{2(4-\alpha)r^2(\sin\theta)^{d}}{(d-1)A^{6-\alpha}}+
\frac{(4-\alpha)(6-\alpha)r^2(\sin\theta)^{d+2}}{(d-1)(d+1)A^{8-\alpha}}\,\md \theta.
\end{align*}
Therefore, by using \eqref{eq17.06}, \eqref{eq13.09}, and the equality above, we obtain
\begin{align}
                                \label{eq16.04}
r\phi'(r)-\phi(r)&=-r^3(2-\alpha)\frac {\omega_{d-1}}{\omega_d}\int_0^\pi \frac {(\sin\theta)^{d-2}}{A^{4-\alpha}}\Big(1-\frac{2(4-\alpha)(\sin\theta)^{2}}{(d-1)A^{2}}\nonumber\\
&\qquad +\frac{(4-\alpha)(6-\alpha)(\sin\theta)^{4}}{(d-1)(d+1)A^{4}}\Big)\,\md \theta.
\end{align}
Since $d+\alpha\ge 5$ and $\alpha<2$, it holds that
$$
1\ge \frac{(4-\alpha)(d+1)}{(d-1)(6-\alpha)},
$$
which yields
\begin{multline*}
1-\frac{2(4-\alpha)(\sin\theta)^{2}}{(d-1)A^{2}}
+\frac{(4-\alpha)(6-\alpha)(\sin\theta)^{4}}{(d-1)(d+1)A^{4}}\\
\ge\frac{4-\alpha}{d-1}\cdot\Big(\frac{d+1}{6-\alpha}-\frac{2(\sin\theta)^{2}}{A^{2}}
+\frac{(6-\alpha)(\sin\theta)^{4}}{(d+1)A^{4}}\Big)
\ge 0.
\end{multline*}
The last inequality above is in fact strict almost everywhere because $A>|\sin\theta|$ except at a single point where $\cos\theta=r$.
Therefore, we get from \eqref{eq16.04} that
$r\phi'(r)-\phi(r)<0$ on $(0,1)\cup (1,\infty)$.
As before, the integral on the right-hand side of \eqref{eq17.06} is absolutely convergent and continuous at $r=1$ by the dominated convergence theorem. Furthermore, it is easy to check that
$$
\lim_{r\to 1} \{r\phi'(r)-\phi(r)\}<0;
$$
See also \eqref{eq12.46} below.
We thus conclude that $r\phi'(r)-\phi(r)<0$ on $(0,\infty)$. The lemma is proved
\end{proof}

We are now in the position to complete the proof of Theorem \ref{thm2.1}.

\begin{proof}[Proof of Theorem \ref{thm2.1}]
With Lemma \ref{lem2.2} in hand, the proof essentially follows an idea in the proof of Theorem 3.1 \cite{Don10}, by observing that, instead of the concavity of $\phi$, the proof of Theorem 3.1 \cite{Don10} only uses the monotonicity property of $\phi/r$.

We use a contradiction argument. Suppose $(\mu_t)$ is a radially symmetric first-kind similarity measure solution such that there are two numbers
$$0<r_1<r_2,\quad r_1,r_2\in\text{supp}\,\hat\mu_0.
$$
Denote $w(r):=|v_0(r e_1)|$. By \eqref{eq16.39},
$$
w(r_k)=\int_0^\infty \phi(r_k/\rho)\rho^{\alpha-1}\,\md \hat \mu_0(\rho),\quad k=1,2.
$$
It follows from Lemma \ref{lem2.2} that
$$
\frac {1}{r_1}\phi(r_1/\rho)\ge \frac {1}{r_2} \phi(r_2/\rho)\quad \forall \rho\in (0,\infty),
$$
with the inequality being strict for $\rho$ in a small neighborhood of $r_1$ because $r_2/r_1>1$. Since $r_1\in \text{supp}\,\hat\mu_0$, we get
$$
\frac {1}{r_1}w(r_1)> \frac {1}{r_2} w(r_2).
$$
However, from \eqref{eq16.08} $w$ is a linear function on $\text{supp}\,\hat\mu_0$:
$$
\frac {1}{r_1}w(r_1)= \frac {1}{r_2} w(r_2).
$$
Therefore, we reach a contradiction. The theorem is proved.
\end{proof}

\mysection{The case $d+\alpha\in (4,5)$}    \label{sec4}

In Theorem \ref{thm2.1}, we proved that double or multi delta-ring similarity solutions to the aggregation equation \eqref{Agg} solution does not exist when $\alpha+d\ge 5$ or $\alpha+d=4$ by showing that any radially symmetric first-kind similarity solution is of the form
\begin{equation*}
\hat \mu_0=m_0\delta_0+m_1\delta_{\rho_1}.
\end{equation*}
In this section we show that this characterization result is still true in the case $\alpha+d\in (4,5)$.
\begin{theorem}
                                    \label{thm4.1}
Let $d\ge 3$ be an integer, $\alpha\in [4-d,2)$, and $\nabla K(x)=x|x|^{\alpha-2}$. Then
any radially symmetric first-kind similarity measure solution is of the form
$$
\mu_t(x)=\frac 1 {R(t)^d} \mu_0\left(\frac x {R(t)}\right),
$$
where
\begin{equation*}
\hat \mu_0=m_0\delta_0+m_1\delta_{\rho_1}
\end{equation*}
for some constants $m_0,m_1\ge 0$ and $\rho_1>0$.
\end{theorem}
In fact, we shall prove that the claim of Lemma \ref{lem2.2} i) still holds if we only assume $d\ge 3$ and $d+\alpha> 4$. As soon as this is verified, one immediately proves Theorem \ref{thm4.1} by using the very same proof of Theorem \ref{thm2.1}. It seems to us that the remaining case $d+\alpha\in (4,5)$ is more involved. One of the difficulties to adapt the proof of Lemma \ref{lem2.2} is as follows. When $d+\alpha\ge 5$, the {\em strict} negativity of $r\phi'-\phi$ is proved by obtaining a complete square in the integrand. This method does not work when $d+\alpha\in (4, 5)$, since in the limiting case $d+\alpha=4$ the function $r\phi'-\phi$ is equivalent to zero for $r\in [0,1]$.

Here we prove Lemma \ref{lem2.2} i) for $d\ge 3$ and $d+\alpha\in (4,5)$ by directly evaluating the coefficients of the series expansion of \eqref{eq11.46}. In the sequel, we always assume $d\ge 3$ and  $d+\alpha\in (4,5)$, and we write $d+\alpha=4-\gamma$ for some $\gamma\in (-1,0)$.

First of all,
we observe that, for any $r\in (0,\infty)$ and $\theta\in (0,\pi)$,
$$
A(r,\theta)=rA(1/r,\theta).
$$
Thus, we have
$$
\int_0^\pi \frac {(\sin\theta)^d} {(A(r,\theta))^{4-\alpha}}\,\md \theta
=r^{\alpha-4}\int_0^\pi \frac {(\sin\theta)^d} {(A(1/r,\theta))^{4-\alpha}}\,\md \theta
$$
and
$$
\int_0^\pi\frac {(\sin\theta)^{d-2}} {(A(r,\theta))^{2-\alpha}}\,\md \theta
=r^{\alpha-2}\int_0^\pi\frac {(\sin\theta)^{d-2}} {(A(1/r,\theta))^{2-\alpha}}\,\md \theta.
$$
Therefore, by \eqref{eq11.46}, to show
$$
\psi(r):=\frac {(d-1)\omega_d}{r(2-\alpha)\omega_{d-1}}\cdot(r\phi'-\phi)<0\quad \text{in}\,\, [0,\infty),
$$
it suffices to prove
\begin{equation}
                                    \label{eq17.36}
\psi(r)<0\quad \forall r\in (0,1].
\end{equation}
We shall prove \eqref{eq17.36} in the remaining part of the section.
Clearly, by using the property of $\Psi_k$ defined in \eqref{eq17.27}, we have $\psi(0)=0$. At $r=1$, we have $A(r,\theta)=2\sin(\theta/2)$ for $\theta\in [0,\pi]$. Thus, we compute
\begin{align}
\psi&(1)=\int_0^\pi \left(d\frac {(\sin\theta)^d} {(2\sin(\theta/2))^{d+\gamma}}-(d-1)\frac {(\sin\theta)^{d-2}} {(2\sin(\theta/2))^{d-2+\gamma}}\right)\,\md \theta\nonumber\\
&=2^{-\gamma}\int_0^\pi \left(d \cos^d\Big(\frac {\theta} 2\Big)-(d-1)\cos^{d-2}\Big(\frac {\theta} 2\Big) \right)\sin^{-\gamma}\Big(\frac {\theta} 2\Big)\,\md \theta\nonumber\\
&=2^{-\gamma}\left(dB\Big(\frac {d+1} 2,\frac {1-\gamma}2\Big)-(d-1)B\Big(\frac {d-1} 2,\frac {1-\gamma} 2\Big)\right)\nonumber\\
                                \label{eq12.46}
&=2^{-\gamma} \frac {\gamma(d-1)}{d-\gamma}B\Big(\frac {d-1} 2,\frac {1-\gamma} 2\Big)<0.
\end{align}

For $r\in (0,1)$, we use change of variables to get
\begin{align*}
\psi(r)&=\int_0^\pi \left(d\frac {(\sin\theta)^d} {A^{d+\gamma}}-(d-1)\frac {(\sin\theta)^{d-2}} {A^{d+\gamma-2}}\right)\,\md \theta\\
&=\int_{-1}^1 \left(\frac {d(1-t^2)^{(d-1)/2}} {(1+r^2-2rt)^{(d+\gamma)/2}}-\frac {(d-1)(1-t^2)^{(d-3)/2}} {(1+r^2-2rt)^{(d+\gamma-2)/2}}\right)\,\md t\\
&=\frac 1 {2r}\int_{(1-r)^2}^{(1+r)^2} \Big(d s^{-(d+\gamma)/2}\big(1-(1+r^2-s)^2/(2r)^2\big)^{(d-1)/2} \\
&\qquad-(d-1) {s^{-(d+\gamma-2)/2}}\big(1-(1+r^2-s)^2/(2r)^2\big)^{(d-3)/2}\Big)\,\md s.
\end{align*}
The last integral can be evaluated explicitly. We expand the outcome in a series with respect to $r$. After a direct but quite lengthy computation, we obtain that, for $d\ge 3$ odd,
\begin{equation}
                            \label{eq14.29}
\psi(r)=-2\sum_{k=1}^\infty \frac{(d-1)!!}{(d+2k)!!(2k-2)!!}a_k r^{2k},
\end{equation}
and for $d\ge 4$ even,
\begin{equation}
                            \label{eq14.30}
\psi(r)=-\pi\sum_{k=1}^\infty \frac{(d-1)!!}{(d+2k)!!(2k-2)!!}a_k r^{2k},
\end{equation}
where
$$
a_1=-\gamma(2-\gamma),\quad a_{k+1}=a_k (2k+\gamma)(d+2k-2+\gamma),\quad k=1,2,\ldots.
$$
It is clear that, for $\gamma\in (-1,0)$, the series above are absolutely convergent in $[0,1)$ and all the coefficients are strictly negative. Therefore, we conclude $\psi(r)<0$ in $(0,1)$ and Theorem \ref{thm4.1} is proved.
\begin{remark}
We note that the argument above does not apply to the case $\gamma\le -1$. Thus it seems unlikely to unify the proofs of Lemma \ref{lem2.2} for different regimes of $\alpha$ in this section and Section \ref{sec2}.
\end{remark}

\mysection{Existence of multi delta-ring solutions}
                                \label{sec3}

In this section, we shall construct multi delta-ring similarity solutions to \eqref{Agg} when $d+\alpha<4$. Let $n\ge 2$ be an integer. The $n$ delta-ring solutions are of the form
$$
\mu_t(x)=\frac 1 {R(t)^d} \mu_0\left(\frac x {R(t)}\right),
$$
where
\begin{equation*}
                                    \label{eq10.40}
\hat \mu_0=\sum_{k=1}^n m_k\delta_{\rho_k}
\end{equation*}
for some positive constants $m_k,\rho_k,k=1,\ldots,n$ satisfying
$$
\sum_{k=1}^n m_k=1,\quad 0<\rho_1<\rho_2<\ldots<\rho_n<\infty.
$$
In the 2D case with $\alpha=1$, the existence of two delta-ring solutions  was obtained in \cite{BCL} by using a fixed point argument. The main result of this section is the following theorem, which gives the existence of $n$ delta-ring solutions for any $n\ge 2$ when $\alpha+d<4$.

\begin{theorem}[Existence of multi delta-ring similarity solutions]
                                        \label{thm3.1}
Let $d\ge 1$ be an integer and $\alpha<2$ such that $\alpha\in (2-d,4-d)$. Then for any integer $n\ge 2$, there exists an $n$ delta-ring solutions to \eqref{Agg}.
\end{theorem}

The lemma below will be used in the proof of Theorem \ref{thm3.1}. We postpone the proof of it to the end of this section.
\begin{lemma}
                                \label{lem2}
Let $d\ge 1$ be an integer and $\alpha<2$ such that $\alpha\in (2-d,4-d)$.
Then we have
\begin{equation}
                                    \label{eq14.39}
\phi(1)>\phi'(0)>0.
\end{equation}
\end{lemma}

\begin{proof}[Proof of Theorem \ref{thm3.1}]
We shall construct a certain type of delta-ring solution where the radii of
rings of the initial data are powers of a
given large number. Our goal is to determine how to distribute the unit mass on these rings, so that the solution will be a similarity solution of the first kind. From the proof below, it should be apparent that for the existence of such a distribution, all we need is that ratios of different radii are sufficiently large (or small). Of course, there could be many other multi delta-ring solutions which are not of this structure.

Let $\lambda$ be a sufficiently large constant to be specified later. We take
$$
\rho_k=\lambda^{k-1},\quad m_k=\rho_k^{2-\alpha}\tilde  m_k,\quad k=1,\ldots,n.
$$
Again, we denote $w(\rho)=|v(\rho e_1)|$. Thanks to \eqref{eq16.08}, in this case a necessary and sufficient condition for $\mu_t$ to be a similarity solution is that
\begin{equation}
                            \label{eq15.12}
w(\rho_k)/\rho_k=w(\rho_j)/\rho_j,\quad k,j=1,\ldots,n.
\end{equation}
So to prove the theorem it suffices to show that, for sufficiently large $\lambda$, there exists $\tilde m_k,k=1,\ldots,n$ such that the condition \eqref{eq15.12} is satisfied.

From \eqref{eq16.39}, we have
$$
w(\rho_k)=\sum_{j=1}^n \phi(\rho_k/\rho_j)\rho_j^{\alpha-1}m_j
=\sum_{j=1}^n \phi(\lambda^{k-j})\rho_j \tilde m_j.
$$
Therefore,
\begin{equation*}
w(\rho_k)/\rho_k=\sum_{j=1}^n \phi(\lambda^{k-j})\lambda^{j-k} \tilde m_j,
\end{equation*}
and \eqref{eq15.12} is equivalent to the existence of a solution $(\mathrm{x}_1,\ldots,\mathrm{x}_n)$ to the linear algebraic equation
\begin{equation}
                            \label{eq16.19}
B_\lambda (\mathrm{x}_1,\ldots,\mathrm{x}_n)^T=(1,\ldots,1)^T
\end{equation}
satisfying $\mathrm{x}_j>0,j=1,\ldots,n$,
where $B_\lambda=[B_{\lambda,kj}]_{j,k=1}^n$ is an $n\times n$ matrix and
$$
B_{\lambda,kj}=\phi(\lambda^{k-j})\lambda^{j-k}\quad j,k=1,\ldots,n.
$$
It follows from \eqref{eq11.30} that
$$
\lim_{\lambda\to \infty}B_{\lambda,kj}=B_{\infty,kj}:=\left\{\begin{aligned}
&\phi'(0) \quad &\text{if}\quad k<j,\\
&\phi(1) \quad &\text{if} \quad k=j,\\
&0 \quad &\text{if}\quad k>j,
\end{aligned}\right.
$$
which together with Lemma \ref{lem2} implies that $B_{\infty}$ is upper-triangular and non-degenerate. It is a simple fact that the equation
$$
B_{\infty}(\mathrm{x}_1,\ldots,\mathrm{x}_n)^T=(1,\ldots,1)^T
$$
has a unique solution which is given by
$$
(\mathrm{x}_1,\ldots,\mathrm{x}_n)^T=\big(\phi(1)\big)^{-1}((1-c)^{n-1},(1-c)^{n-2},\ldots,1)^T,
$$
where $c=\phi'(0)/\phi(1)$. By Lemma \ref{lem2}, we have $1-c>0$. Since $B_\infty$ is non-degenerate, by a standard continuity argument, for $\lambda>0$ sufficiently large, \eqref{eq16.19} has a unique solution, which satisfies $\mathrm{x}_j>0,j=1,\ldots,n$. The theorem is thus proved.
\end{proof}

We finish this section by giving the proof of Lemma \ref{lem2}.
\begin{proof}[Proof of Lemma \ref{lem2}]
First we remark that \eqref{eq14.39} essentially follows from the calculations in Section \ref{sec4}. Indeed, by the expansions \eqref{eq14.29} and \eqref{eq14.30}, one can easily see that $\psi>0$ in $(0,1)$, which implies that $\phi(r)/r$ is strictly increasing in $(0,1]$.
Below we give an alternative proof without using the complicated series expansions.

As before, denote $\gamma=4-d-\alpha\in (0,2)$. By \eqref{eq17.02}, we have
\begin{align}
\phi'(0)&=\frac {\omega_{d-1}}{\omega_d}\int_0^\pi (\sin\theta)^{d-2}\big(1 -(2-\alpha)(\cos\theta)^2\big)\,\md \theta\nonumber\\
&=\frac {\omega_{d-1}}{\omega_d}\Big[B\Big(\frac {d-1} 2,\frac 1 2\Big)
-(2-\alpha)B\Big(\frac {d-1} 2,\frac 3 2\Big)\Big]\nonumber\\
&=\frac {\omega_{d-1}}{\omega_d}\frac {2-\gamma} d B\Big(\frac {d-1} 2,\frac 1 2\Big)\nonumber\\
                                            \label{eq15.23}
&=\frac {\omega_{d-1}}{\omega_d}\frac {2-\gamma} d \Gamma\Big(\frac {d-1} 2\Big)
\Gamma\Big(\frac 1 2\Big)\Big(\Gamma\Big(\frac d 2\Big)\Big)^{-1}>0,
\end{align}
where $B(\cdot,\cdot)$ is the Beta function.
Here, we used the formula
$$
\int_0^{\pi/2}(\cos\theta)^\beta (\sin\theta)^\gamma\,\md \theta=\frac 1 2B\Big(\frac {\beta+1} 2,\frac{\gamma+1} 2\Big)\quad \forall\,\,\beta,\gamma>-1.
$$
By \eqref{eq16.59}, we compute
\begin{align}
\phi(1)&=\frac {\omega_{d-1}}{\omega_d}\int_0^\pi \frac {(1-\cos\theta)(\sin\theta)^{d-2}}{(2-2\cos\theta)^{1-\alpha/2}}\,\md \theta\nonumber\\
&=\frac {\omega_{d-1}}{\omega_d}\int_0^\pi \frac 1 2 \big(2\sin(\theta/2)\cos(\theta/2)\big)^{d-2}
\big(2\sin(\theta/2)\big)^{\alpha}\,\md \theta\nonumber\\
&=\frac {\omega_{d-1}}{\omega_d}2^{1-\gamma}\int_0^\pi \big(\sin(\theta/2)\big)^{2-\gamma}
\big(\cos(\theta/2)\big)^{d-2}\,\md \theta\nonumber\\
&=\frac {\omega_{d-1}}{\omega_d}2^{1-\gamma}B\Big(\frac {3-\gamma} 2,\frac{d-1} 2\Big)\nonumber\\
                                        \label{eq15.13}
&=\frac {\omega_{d-1}}{\omega_d}2^{1-\gamma}\Gamma\Big(\frac {3-\gamma} 2\Big)
\Gamma\Big(\frac{d-1} 2\Big)\Big(\Gamma\Big(\frac {d-\gamma+2} 2\Big)\Big)^{-1}.
\end{align}
It follows from the duplication formula for the Gamma function that
$$
\Gamma\Big(\frac {3-\gamma} 2\Big)=2^{\gamma-1}
\Gamma\Big(\frac 1 2\Big)\Gamma\Big(2-\gamma\Big)\Big(\Gamma\Big(\frac {2-\gamma} 2\Big)\Big)^{-1}.
$$
This together with \eqref{eq15.13} gives
\begin{equation}
                                            \label{eq15.36}
\phi(1)=\frac {\omega_{d-1}}{\omega_d}\Gamma\Big(\frac 1 2\Big)\Gamma\Big(2-\gamma\Big)
\Gamma\Big(\frac{d-1} 2\Big)\Big(\Gamma\Big(\frac {d-\gamma+2} 2\Big)\Gamma\Big(\frac {2-\gamma} 2\Big)\Big)^{-1}.
\end{equation}
By using \eqref{eq15.23} and \eqref{eq15.36}, to prove \eqref{eq14.39} it suffices to show
$$
\Gamma\Big(2-\gamma\Big)
\Big(\Gamma\Big(\frac {d-\gamma+2} 2\Big)\Gamma\Big(\frac {2-\gamma} 2\Big)\Big)^{-1}
>\frac {2-\gamma} d
\Big(\Gamma\Big(\frac d 2\Big)\Big)^{-1},
$$
which is equivalent to
\begin{equation}
                                    \label{eq15.46}
\Gamma\Big(2-\gamma\Big)\Gamma\Big(\frac {d+2} 2\Big)>
\Gamma\Big(\frac {d-\gamma+2} 2\Big)\Gamma\Big(\frac {4-\gamma} 2\Big).
\end{equation}
Since
$$
2-\gamma+\frac {d+2} 2=\frac {d-\gamma+2} 2+\frac {4-\gamma} 2
$$
and
$$
(2-\gamma)\cdot\frac {d+2} 2<\frac {d-\gamma+2} 2\cdot\frac {4-\gamma} 2,
$$
the inequality \eqref{eq15.46} is an immediate consequence of the log-convexity property of the Gamma function.
\end{proof}

\mysection{An example}
                                \label{sec5}
Finally, we present an example of first-kind similarity solutions with support on open sets in the limiting case $\alpha=2-d$, which generalizes a 1D example constructed in \cite{BCL} with $\alpha=1$.

\begin{example}
                                \label{ex5.1}
Assume $d\ge 2$ and $\alpha=2-d$. Let $u_0(x)=d\omega_d^{-1}  I_{|x|<1}$. Then $u$ given by \eqref{eq00.25} with $R(t)=(1-t\omega_d)^{1/d}$ is a measure similarity solution to the aggregation equation \eqref{Agg}. Moreover, $T_0=\omega_d^{-1}$ is the blowup time of the solution.
\end{example}

As mentioned before, to prove that $u$ is a similarity solution, it suffices to show that $$
w(r):=\int_0^\infty \phi(r/\rho)\rho^{\alpha-1}\,\md \hat \mu_0(\rho)
$$
is a linear function with respect to $r$ on $[0,1)$, which is the support of $\hat \mu_0$. For $\alpha=2-d$, from \eqref{eq16.59} we have
$$
\phi(r)=\frac {\omega_{d-1}}{\omega_d}\int_0^\pi \frac {(r-\cos\theta)(\sin\theta)^{d-2}}{A^{d}(r,\theta)}\,\md \theta,
$$
which can be evaluated explicitly:
$$
\phi(r)= \left\{\begin{aligned}
&0 \quad &\text{on}\quad [0,1),\\
&1/2 \quad &\text{at}\quad r=1,\\
&r^{1-d} \quad &\text{on} \quad (1,\infty).
\end{aligned}\right.
$$
Consequently, for $r\in [0,1)$, we have
\begin{align*}
w(r)&:=\int_0^1 \phi(r/\rho)\rho^{\alpha-1}\,\md \hat\mu_0(\rho)\\
&=\int_0^r \phi(r/\rho)\rho^{\alpha-1}\,\md \hat\mu_0(\rho)\\
&=r^{1-d}\mu_0(B_r)=\frac{\omega_d} d r.
\end{align*}
Therefore, indeed $w$ is a linear function on $[0,1)$ and $u$ is a similarity solution to \eqref{Agg}. Furthermore, we have $\lambda=\omega_d/d$. Following the argument at the beginning of Section \ref{sec2}, we obtain
$$
T_0=(2-\alpha)^{-1}\lambda^{-1}=\omega_d^{-1},
\quad R(t)=(1-t\omega_d)^{1/d}.
$$

\section*{Acknowledgement}
The author would like to thank Thomas Laurent for his helpful discussions and comments. He also wishes to thank the referees for valuable comments and suggestions.

\end{document}